\documentclass[10pt]{amsart}
\usepackage{amsmath}
\usepackage{amsfonts}
\usepackage{amssymb}
\usepackage{amsthm}
\usepackage{graphicx}

\usepackage{hyperref}

\usepackage{color}

\def \de {\partial}
\def \e {\varepsilon}

\def \O {\Omega}
\def \o {\omega}
\def \phi {\varphi}
\def \RN {\mathbb{R}^N}
\def \SN {S^{N-1}}
\def \SNp {S^{N-1}_+}
\def \Speta {S^{N-1}_+(\eta)}

\def \R {\mathbb{R}}

\def \div {\text{div}}

\def \Swr {S_{\omega, R}}
\def \Sw {\Sigma_\omega}
\def \Swp {\Sigma_{\omega'}}
\def \Swpn {\Sigma_{\omega_n'}}
\def \Wup {W_0^{1,p}(\O\cup\Gamma_1)}
\def \Wud {W_0^{1,2}(\O\cup\Gamma_1)}
\def \TwO {\mathcal{T}_\o(\O)}
\def \Petar {\Pi_+(\eta,r)}

\newtheorem{theorem}{Theorem}[section]
\newtheorem{lemma}[theorem]{Lemma}
\newtheorem{proposition}[theorem]{Proposition}
\newtheorem{corollary}[theorem]{Corollary}
\newtheorem*{theoPt1}{Theorem A [Theorem 1.1 in (\cite{PT})]}

\newtheorem{remark}[theorem]{Remark}

\theoremstyle{definition}
\newtheorem{definition}[theorem]{Definition}

\numberwithin{equation}{section}

\begin{document}

\title[Isoperimetric cones and minimal solutions of partial overdetermined problems]{Isoperimetric cones and\\ minimal solutions of partial overdetermined problems}
\author[F. Pacella]{Filomena Pacella}
\address{Dipartimento di Matematica, Sapienza Universit\`a di Roma, P.le Aldo Moro 5 - 00185 Roma, Italy.
        }
 \email{pacella@mat.uniroma1.it}

\author[G. Tralli]{Giulio Tralli}
\address{Dipartimento d'Ingegneria Civile e Ambientale (DICEA), Universit\`a di Padova, Via Marzolo 9 - 35131 Padova,  Italy.
         }
 \email{giulio.tralli@unipd.it}


\date{}

\begin{abstract}
In this paper we consider a partial overdetermined mixed boundary value problem in domains inside a cone as in \cite{PT}. We show that in cones having an isoperimetric property the only domains which admit a solution and which minimize a torsional energy functional are spherical sectors centered at the vertex of the cone. We also show that cones close in the $C^{1,1}$-metric to an isoperimetric one are also isoperimetric, generalizing so a result of \cite{BF}. This is achieved by using a characterization of constant mean curvature polar graphs in cones which improves a result of \cite{PT}.
\end{abstract}
\maketitle

\section{Introduction}\label{intro}

In this paper we complement and extend some results recently obtained in \cite{PT} about partially overdetermined problems in bounded domains in cones and about constant mean curvature surfaces in cones satisfying suitable gluing conditions.\\
Let $\omega$ be an open connected set on the unit sphere $\SN$, $N\geq 2$, and let us denote by $\Sw$ the open cone in $\RN$ with vertex at the origin $O$ given by
$$\Sw=\{tx\,:\,x\in\omega,\,\,t\in(0,+\infty)\}.$$
We will assume that $\de\Sw$ is Lipschitz-continuous. As in \cite{PT} we consider a sector-like domain in $\Sw$ which is a bounded domain $\O\subset\Sw$ whose boundary is Lipschitz-continuous and is given by
$$\de\O=\Gamma\cup\Gamma_1\cup\de\Gamma$$
where $\Gamma$ is the relative (to $\Sw$) boundary, i.e. $\Gamma$ is the part of $\de\O$ which is contained in $\Sw$, $\Gamma_1=\de\O\smallsetminus\overline{\Gamma}$ and $\de\Gamma=\de\Gamma_1=\overline{\Gamma}\cap\overline{\Gamma}_1$. We require that $\textsc{H}_{N-1}(\Gamma)>0$, $\textsc{H}_{N-1}(\Gamma_1)>0$ where $\textsc{H}_{N-1}(\cdot)$ denotes the $(N-1)$-dimensional Hausdorff measure.\\
Particular cases of sector-like domains are the spherical sectors centered at the vertex of the cone, we denote them by $\Swr$ i.e. 
$$\Swr=B_R\cap\Sw,\qquad R>0$$
where $B_R$ is the ball of radius $R$ centered at the origin.\\
Then we consider the partially overdetermined problem
\begin{equation}\label{serrintype}
\begin{cases}
   -\Delta u=1 & \mbox{ in }\O, \\
   u=0 & \mbox{ on }\Gamma,\\
	 \frac{\de u}{\de\nu}=-c<0 & \mbox{ on }\Gamma,\\
	\frac{\de u}{\de\nu}=0 & \mbox{ on }\Gamma_1\smallsetminus\{O\},
\end{cases}
\end{equation}
where $\nu$ denotes the exterior unit normal to $\de\O$ whenever is defined. We will sometimes write $\nu_x$ for $x\in\de\Gamma$, meaning that $\nu_x$ is the normal to $\overline{\Gamma}$.\\
If $\de\Sw\smallsetminus\{O\}$ and $\Gamma$ are smooth hypersurfaces, then the following result is proved in \cite{PT}
\begin{theoPt1}
Let $c>0$ be fixed, and consider a convex cone $\Sw$ such that $\de\Sw\smallsetminus\{O\}$ is smooth. Assume that $\O$ is a sector-like domain having a smooth relative boundary $\Gamma$ with smooth $\de\Gamma\subset \de\Sw\smallsetminus\{O\}$. If there exists a classical $C^2(\O)\cap C^1(\Gamma\cup \Gamma_1\smallsetminus\{O\})$-solution $u$ of problem \eqref{serrintype} such that $u\in W^{1,\infty}(\O) \cap W^{2,2}(\O)$, then
$$\O=\Sw\cap B_{R}(p_0),\quad \mbox{ and }\quad u(x)=\frac{N^2c^2-|x-p_0|^2}{2N},$$
where $B_{R}(p_0)$ denotes the ball centered at a point $p_0\in\RN$ and radius $R=Nc$.\\
Moreover, one of the following two possibilities holds:
\begin{itemize}
\item[(i)] $p_0=O$;
\item[(ii)] $p_0\in\de\Sw$ and $\Gamma$ is a half-sphere lying over a flat portion of $\de\Sw$.
\end{itemize}
\end{theoPt1}
Let us observe that the hypothesis that the solution $u$ belongs to $W^{1,\infty}(\O) \cap W^{2,2}(\O)$ is automatically satisfied when $\Gamma$ and $\Gamma_1$ intersect orthogonally, as proved in \cite{PT}. We also refer the reader to the recent works \cite{CR, GX, LS}.\\
The previous theorem gives a characterization of sector-like domains $\O$ in which a solution of the partially overdetermined problem \eqref{serrintype} exists. The claim is that either $\O$ is a spherical sector centered at the vertex of the cone or is a half ball centered at a point $\de\Sw\smallsetminus\{O\}$. The last case can happen only when $\de\Sw$ has a flat portion. One of the aims of this paper is to show a connection between problem \eqref{serrintype} and a suitable torsional energy function $\TwO$ that can be defined for sector-like domains (see \eqref{defT}). We prove that in any smooth cone the domains $\O$ which are stationary for $\TwO$, under a volume constraint, are the ones for which \eqref{serrintype} admits a weak solution (see Proposition \ref{domder}). Consequently, if the cone $\Sw$ is convex the stationary sector-like domains which are smooth enough can be characterized by Theorem $A$.\\
Moreover we show that a conical version of the classical Saint-Venant principle (see, e.g., \cite{Ke}) holds. More precisely we prove in Theorem \ref{saintV} that if the cone has an isoperimetric property (but is neither necessarily convex nor smooth) then the only sector-like domains which minimize $\TwO$ under a volume constraint are the spherical sectors $\Swr$. These results are analogous to those holding for the classical torsional rigidity problem (see \cite{PS}). The proof of Theorem \ref{saintV} is easily obtained by using the $\omega$-symmetrization which is well defined in isoperimetric cones (see Section \ref{sec3}). This implies that Theorem A can be extended to the class of isoperimetric cones (which not only includes convex cones, see Section \ref{sec3}) relatively to the characterization of the sector-like domains $\O$ which admit a solution of problem \eqref{serrintype} and which also minimize the functional $\TwO$. In this case, up to rescaling, the only domain is the spherical sector $S_{\omega, 1}$, i.e. the alternative $ii)$ of Theorem $A$ does not hold.\\
By an isoperimetric cone we mean a cone which has the property that the only sets which minimize the relative (to $\Sw$) perimeter under a volume constraint are the spherical sectors $\Swr$ (see Definition \ref{defisocone}). It was proved in \cite{LP} that any smooth convex cone is isoperimetric (see also \cite{CRS, FI, RitRos}).\\Recently Baer and Figalli (\cite{BF}) have extended the isoperimetric property to almost convex cones satisfying an uniform $C^{1,1}$ condition.\\
Here we generalize the result of \cite{BF} by proving that the convexity of the cone is not needed in the sense that any cone close to an isoperimetric cone is also isoperimetric (see Theorem \ref{thsect3}). In other words the set of the isoperimetric cones is open with respect to the $C^{1,1}$-distance on $\SN$.\\
Moreover we shorten considerably the proof given in \cite{BF}. Indeed the proof of Theorem 1.2 in \cite{BF} is made in two steps. The first one consists in showing that in the almost convex cones the relative boundary of the minimizers are $C^1$-polar graphs. Then the second step aims to prove that if the relative boundary is a polar graph then the minimizer is actually a spherical sector. The second step is achieved by means of a refined Poincar\'e inequality obtained in convex cones and is the longest part of the proof. To prove Theorem \ref{thsect3} we observe that, in order to reduce to consider only minimizers whose boundary is a polar graph, the convexity of the limit cone is not needed but is enough to require it to have the isoperimetric property (see the details in Section \ref{sec3}). Then we just use the characterization of constant mean curvature polar graphs (or equivalently strictly starshaped hypersurfaces) intersecting any cone orthogonally provided by Theorem \ref{prth1} below to conclude that in almost isoperimetric cones the only minimizers are spherical sectors.

\begin{theorem}\label{prth1}
Let $\Sw$ be any cone in $\RN$ such that $\omega$ is strictly contained in $\SNp=\{x=(x_1,\ldots,x_N)\,:\,x_N>0,\,|x|=1\}$ and it has $C^{1,1}$-smooth boundary. Let $\Gamma\subset \Sw$ be a smooth $(N-1)$-dimensional manifold which is relatively open, bounded, orientable, connected and with $C^{1,1}$-smooth boundary $\de\Gamma$ contained in $\de\Sw\smallsetminus\{O\}$. Assume that $\Gamma$ and $\de\Sw$ intersect orthogonally at the points of $\de \Gamma$ and that the mean curvature of $\Gamma$ is a constant $H>0$. If $\Gamma$ is strictly starshaped with respect to $O$, then $\Gamma$ is the (relative to $\Sw$)-boundary of the spherical sector $S_{\omega, \frac{1}{H}}$.
\end{theorem}
This theorem is an improvement of Theorem 6.4 in \cite{PT} which, in turn, is a particular case of Theorem 1.3 of \cite{PT} where a more general gluing condition between $\Gamma$ and $\Sw$ is assumed. The differences between Theorem \ref{prth1} and Theorem 6.4 of \cite{PT} rely on the regularity assumptions on $\omega$ and on the fact that in \cite{PT} it was proved that $\Gamma=\de B_{\frac{1}{H}}(p_0)\cap\Sw$ for some $p_0\in\de\Sw$, but we could not claim that $p_0$ is actually the vertex of the cone, while this is asserted in Theorem \ref{prth1}. The problem of identifying the center of the sphere on which $\Gamma$ lies is studied in Section \ref{sec2}. We also observe inside the proof of Theorem \ref{prth1} that $H>0$ is a necessary condition, and hence it is not really an hypothesis.

The paper is organized as follows. In Section \ref{sec2} we prove Propositions \ref{lemmaconvexstr} and \ref{eccop0perpolar}, together with other geometric properties needed for the proof of Theorem \ref{prth1}. In Section \ref{sec3} we define the isoperimetric cones and prove in Theorem \ref{thsect3} the generalization of the result of \cite{BF}. In the same section we recall the $\omega$-symmetrization in isoperimetric cones and show the analogous of the Polya-Szego inequality with the characterization of the equality case. Finally in Section \ref{sec4} we study the torsional energy functional and prove the characterization of its stationary points, as well as the Saint-Venant type principle in isoperimetric cones.

\section{Some geometric results}\label{sec2}

In this section, to the aim of proving Theorem \ref{prth1}, we study the following geometrical question:
\begin{center}
\vskip 0.3cm
\noindent\emph{Let $\Sw$ be a cone in $\RN$, $N\geq3$, and assume that $\Gamma$ is a portion of a sphere inside $\Sw$, centered at a point $p_0\in\RN$, i.e. $\Gamma=\de B_{R}(p_0)\cap\Sw$. Assume further that $\Gamma$ and $\de\Sw$ intersect orthogonally at every point of $\de\Gamma\cap\de^*\Sw$, where $\de^*\Sw$ denotes the set of regular points of $\de\Sw$. Can we claim that $p_0$ must be $O$, i.e. the vertex of the cone?}
\end{center}
\vskip 0.3cm
\noindent In the paper \cite{RitRos} it is proved that if $\Sw$ is a smooth convex cone then either $p_0=O$, or $p_0\in \de\Sw\smallsetminus\{O\}$ and $\Gamma$ is a half-sphere lying over a flat portion of $\de\Sw$. Thus, in particular, if the cone is strictly convex, the answer to the question is affirmative.\\
However, this cannot be used in Theorem \ref{prth1} since no convexity assumptions on the cone are made. In Proposition \ref{lemmaconvexstr} we prove, in particular, that it is enough to have a point $\bar{x}\in\de\Gamma$ of strict convexity for $\de\Sw$ to get that the center $p_0$ of the sphere is the vertex of the cone. Moreover, if $\Gamma$ is a polar graph on a $C^1$-domain $\omega$, we prove in Proposition \ref{eccop0perpolar} that $p_0=O$ unless $\omega$ is an half-sphere (i.e. $\Sw$ is an half-plane).

We start by fixing some notations. For a cone $\Sw$ we denote by $\de^*\Sw$ the set of smooth points of $\de\Sw$, i.e. points where $\de\Sw$ is of class $C^1$. In particular we have $\de^*\Sw=\de\Sw\smallsetminus\{O\}$ if $\o$ is a $C^1$-domain. We note that $x\in \de^*\Sw$ iff $\lambda x\in\de^*\Sw$ for every $\lambda >0$. Moreover, whenever $M$ is a manifold locally $C^1$ around $x\in M$, we denote by $T_x M$ the tangent space of $M$ at $x$. If $M$ is of codimension 1, we denote by $\nu_x^M$ a choice of the unit normal at $x$ (the outward choice, if $M$ is the boundary of a bounded set).\\
In the case $M=\de\Sw$, it is helpful to have in mind that the following facts hold true for all $x\in\de^*\Sw$:\, $T_x\de\Sw$ coincides with the affine space $x+T_x\de\Sw$, and the halfline $\left\langle x\right\rangle^+:=\{\lambda x\,:\,\lambda>0\}$ is contained in $\de\Sw\cap T_x\de\Sw$.

\begin{lemma}\label{lemmatang}
Let $\Gamma$ be the portion of a sphere inside a cone $\Sw$, i.e. $\Gamma= \de B_{r}(p_0)\cap \Sw$ where $B_R(p_0)$ denotes the ball of radius $r>0$ and center $p_0\in\RN$. If $\Gamma$ intersects $\de\Sw$ orthogonally at a point $x\in\de\Gamma \cap \de^*\Sw$, then
$$p_0\in T_x\de\Sw.$$
\end{lemma}
\begin{proof}
The fact that $\Gamma$ intersects $\de \Sw$ orthogonally at $x$ says that the scalar product $\left\langle \nu_x^{\overline{\Gamma}}, \nu_x^{\de\Sw}\right\rangle=0$ where $\nu_x^{\overline{\Gamma}}$ denotes the normal to $\overline{\Gamma}$ at $x$ and $\nu_x^{\de\Sw}$ the normal to $\de\Sw$ at $x$. Since $\nu_x^{\overline{\Gamma}}=\frac{x-p_0}{r}$ because $\Gamma$ is a portion of a sphere and $\left\langle x, \nu_x^{\de\Sw}\right\rangle=0$ by the cone property, the statement readily follows.
\end{proof}

\begin{remark}\label{rempyr}
The previous lemma, although very simple, allows to locate $p_0$ in some situations. We give a couple of examples, in the case when $\Gamma$ intersects $\de\Sw$ orthogonally at every point of $\de\Gamma\cap \de^*\Sw$.\\
We can consider the case of two distinct hyperplanes intersecting on a $(N-2)$-dimensional space $l$. These define a cone $\Sigma$ for which $\de^*\Sigma=\de\Sigma\smallsetminus l$. There are two possibilities: either $\de\Gamma$ touches both the hyperplanes and then $p_0\in l$ since it must belong to the intersection of all tangent planes which are just the two given hyperplanes, or $\de\Gamma$ touches only one of the two hyperplanes and then $p_0$ lies on the same hyperplane and it is easy to see that $\Gamma$ is forced to be a half-sphere because of the orthogonality condition.\\
We can also consider the case of a pyramid-shaped cone, i.e. a cone constructed by a collection of a finite number of hyperplanes (facets) intersecting just at $O$. In such case, the previous lemma tells us that $p_0=O$ if $\de\Gamma$ touches more than two facets, otherwise the same two alternatives of the previous case arise.
\end{remark}

Denote by $\de^{*2}\Sw$ the set of points $x\in\de\Sw$ around which $\de\Sw$ is of class $C^2$. We recall that, for $x\in \de^{*2}\Sw$, the second fundamental form $h_x$ of $\de\Sw$ is the bilinear symmetric form on $T_x\de\Sw\times T_x\de\Sw$ which can be defined on a orthonormal frame $\{e_1,\ldots,e_{N-1}\}$ as
$$h_x(e_i,e_j)=\left\langle \nabla_{e_i}\nu_x^{\de\Sw},e_j\right\rangle,\qquad \mbox{for }i,j\in \{1,\ldots,N-1\}.$$
Moreover the cone has the property that $\left\langle x,\nu^{\de\Sw}_x\right\rangle=0$ for any $x\in\de^*\Sw$. Thus, the radial direction is not only a tangent direction for $\de\Sw$, but it is also a direction of complete flatness in the sense that $h_x(x,\cdot)\equiv 0$. As a matter of fact, for any tangent direction $e\in T_x\de\Sw$, we have $h_x(x,e)=-\left\langle\nu_x^{\de\Sw},\nabla_e x \right\rangle=-\left\langle \nu_x^{\de\Sw},e\right\rangle=0$.\\
If $N\geq 3$ we give the following definitions:
\begin{itemize}
\item[-] we say that a point $x\in \de^{*2}\Sw$ is \emph{transversally nondegenerate} if the quadratic form $h_x$ restricted to the tangent directions to $\de\Sw$ which are orthogonal to $x$ has all the eigenvalues different from zero. In other words, all principal curvatures of $\de\Sw$ at $x$ are non-zero except for that in the $x$-direction;
\item[-] we say that a point $x\in \de^{*2}\Sw$ is a point of strict convexity (resp. strict concavity) for $\de\Sw$ if the quadratic form $h_x$ is strictly positive (resp. strictly negative) definite when it is restricted to the tangent directions which are orthogonal to $x$.
\end{itemize}

\begin{proposition}\label{lemmaconvexstr}
Let $N\geq 3$. Consider a portion of a sphere $\Gamma= \de B_{r}(p_0)\cap \Sw$ which intersects orthogonally $\de \Sw$ at every point of $\de\Gamma \cap\de^*\Sw$. Suppose there exists a point $\bar{x}\in \de\Gamma\cap\de^{*2}\Sw$ which is transversally nondegenerate. Then
$$p_0=O.$$
\end{proposition}
\begin{proof} Let us split the proof in two steps, and assume the point $\bar{x}\in \de\Gamma\cap\de^{*2}\Sw$ is transversally nondegenerate.\\
\noindent{{\bf{\it Step I. }}}We claim that there exist a point $\tilde{x}\in\de\Gamma\cap\de^{*2}\Sw$ and an open neighborhood $U_{\tilde{x}}$ of $\tilde{x}$ in $\de^{*2}\Sw$ such that every point $x\in U_{\tilde{x}}$ is transversally nondegenerate and such that $V:=\{t x\,:\,x\in U_{\tilde{x}}\cap\de\Gamma,\,t\in(\frac{1}{2},\frac{3}{2})\}$ is an open neighborhood (relatively to $\de\Sw$) of $\tilde{x}$ in $\de\Sw$.\\
To prove this, we notice that in a small neighborhood of $\bar{x}$ we can find a point $\tilde{x}\in \de\Gamma\cap\de^{*2}\Sw$ which is transversally nondegenerate and such that $\tilde{x}\notin T_{\tilde{x}}\de\Gamma$. In fact, if this was not true then in a neighborhood of $\bar{x}$ the manifold $\de\Gamma$ (which is ($N-2$)-dimensional) would contain a straight segment and this is not possible since $\de\Gamma\subset \de B_{r}(p_0)$. By continuity, also at the points close to $\tilde{x}$ the radial direction has non-vanishing component which is transversal to $\de\Gamma$: this says that an open tubular neighborhood of $\de\Gamma$ in $\de^{*2}\Sw$ around $\tilde{x}$ is contained in $V$.\\
\noindent{{\bf{\it Step II. }}} We now complete the proof of the lemma.
Take the point $\tilde{x}$ whose existence is guaranteed by Step I. We can write an orthonormal frame for $T_x\de\Sw$, for $x$ in an open neighborhood $V_1\subseteq V$, as $\{\frac{x}{|x|},e_1,\ldots,e_{N-2}\}$. We can always pick the $e_j$'s such that they diagonalize the second fundamental form $h_x$. By the nondegeneracy property we have $h_x(e_j,e_j)=\lambda_j(x)\neq 0$ for any $j\in\{1,\ldots,N-2\}$. On the other hand, we know from Lemma \ref{lemmatang} that $\left\langle p_0,\nu^{\de\Sw}_x \right\rangle=0$. By differentiating such relation in $V_1$ along $e_j$, we get
$$0=e_j\left(\left\langle p_0,\nu_x \right\rangle\right)=\lambda_j(x)\left\langle p_0,e_j\right\rangle$$  
for all $j$. Therefore, the tangent vector $p_0$, being orthogonal to every $e_j$, has to be parallel to $x$. Since this holds true for any $x$ in an open (i.e. ($N-1$)-dimensional) neighborhood of $\tilde{x}$, $p_0$ is then forced to be $O$.
\end{proof}

\begin{remark} By the previous proposition we deduce that if $p_0\neq O$ then at all points of $\de\Gamma\cap\de^*\Sw$ some of the principal curvatures in the directions orthogonal to $x$ must vanish. Some cases when this happens are those described in Remark \ref{rempyr}.
\end{remark}

Let us turn our attention on $C^1$-polar graph. In this case we can identify completely the point $p_0$.

\begin{proposition}\label{eccop0perpolar}
Consider a cone $\Sw$ such that $\de\omega$ is $C^1$-smooth. Suppose $\overline{\Gamma}=\overline{\de B_{r}(p_0)\cap \Sw}$ is a strictly starshaped hypersurface with respect to $O$ which intersects orthogonally $\de\Sw$ at every point of $\de\Gamma\subset \de\Sw\smallsetminus\{O\}$. Then one of the following two possibilities holds:
\begin{itemize}
\item[(i)] $p_0=O$;
\item[(ii)] $p_0\in\de\Sw$ and $\Sw$ is an half-space.
\end{itemize}
\end{proposition}
\begin{proof}
Arguing as in \cite[Proof of Lemma 4.10]{RitRos} (see also the details given in \cite[Lemma 2.4, Step I and II]{PT}) one can prove that $p_0\in\de\Sw$. Suppose that $p_0\neq O$, and denote $q_0=\frac{p_0}{|p_0|}$. Since $\overline{\Gamma}$ is a polar graph over $\overline{\omega}$, the points of $\de\Gamma$ are in correspondence with the points of $\de\omega$. Recalling that $T_x\de\Sw=T_{\lambda x}\de\Sw$ for any $\lambda >0$, we have from Lemma \ref{lemmatang} that
$$\de\omega\ni q_0\in T_x\de\Sw \quad\mbox{ for all }x\in\de\omega.$$
Therefore, the vector field
$$V_x=q_0-\left\langle q_0,x\right\rangle x,\quad\mbox{ for }x\in\SN,$$
is tangent to $\de\omega$ at every point $x\in\de\omega$. The vector field $V_x$ vanishes only at the antipodal points $\pm q_0$. Moreover, the flow lines of $V$ are the great circles passing through $q_0$: here, according to the standard notation (see, e.g. \cite[pg. 137]{Lee}), by great circle we mean the intersection of the sphere with a $2$-dimensional linear subspace. As a matter of fact it is easy to see that, for any $\bar{x}\in \SN$ different from $\pm  q_0$, the flow line of $V$ starting at $\bar{x}$ is given by the great circle ${\rm span}\{q_0,\bar{x}\}\cap \SN$. Then $\de\omega$, which is a closed $(N-2)$-dimensional $C^1$ manifold, contains all great circles passing through $q_0$ and any point $x\in\de\omega$. Observing that any great circle belongs to the boundary of an half sphere, we deduce that $\de\omega$ must be itself the boundary of an half-sphere (which is indeed $\omega$). Hence $\Sw$ is an half-space, and this concludes the proof.
\end{proof}

We conclude this section proving Theorem \ref{prth1}.

\begin{proof}[Proof of Theorem \ref{prth1}]
The proof consists of two parts. We first show that 
\begin{equation}\label{claimJel}
\exists p_0\in\RN\,\mbox{ such that }\Gamma=\de B_r(p_0)\cap\Sw,
\end{equation}
and then we conclude that $\Gamma$ is in fact a spherical sector by using Proposition \ref{eccop0perpolar}.\\
The assertion \eqref{claimJel} is already proved in \cite[Theorem 6.4]{PT}, where it is deduced from the more general \cite[Theorem 1.3]{PT}. Here we show a self-contained and more direct proof of \eqref{claimJel} holding when $\overline{\Gamma}$ and $\de\Sw$ intersect orthogonally and $\de\omega$ is $C^{1,1}$-smooth. We start by recalling the relation
\begin{equation}\label{Muno}
\div_\Gamma\left(x-\left\langle x,\nu^\Gamma\right\rangle\nu^\Gamma\right)=(N-1)-(N-1)H\left\langle x,\nu^\Gamma\right\rangle,
\end{equation}
which holds for all $x$ belonging to the smooth hypersurface $\Gamma$. Moreover, the vector field $F(x)=x-\left\langle x,\nu^\Gamma\right\rangle\nu^\Gamma$ is $Lip$-smooth up the boundary since we are assuming that $\Gamma$ is $C^{1,1}$-smooth up to the boundary. The orthogonality assumption (and the fact that $\Sw$ is a cone) tells us that
$$\left\langle F(x), n_x\right\rangle =0\qquad\forall x\in\de\Gamma$$
since the outward unit conormal $n_x$ to $\de\Gamma$ coincides in fact with $\nu^{\de\Sw}$. Hence, by integrating \eqref{Muno} over $\Gamma$, we get the first Minkowski formula
\begin{equation}\label{unoM}
\int_\Gamma \left(1 - H\left\langle x,\nu^\Gamma\right\rangle\right){\rm d}\sigma = 0.
\end{equation}
Incidentally, since $H$ is constant, we also deduce that $H$ is necessarily positive because we have $H=\frac{|\Gamma|}{\int_\Gamma \left\langle x,\nu^\Gamma\right\rangle}=\frac{|\Gamma|}{N|\O|}>0$. This is why we have assumed $H>0$ from the beginning. It is proved in \cite[Proposition 1]{CP} that also higher order Minkowski formulas hold true under the orthogonality assumption. In particular we have the validity of the following second Minkowski formula
\begin{equation}\label{dueM}
\int_\Gamma \left(H - \sigma_2(h)\left\langle x,\nu^\Gamma\right\rangle\right){\rm d}\sigma = 0,
\end{equation}
where $\sigma_2(h)$ denotes the second elementary symmetric function of the eigenvalues of $h$, i.e. $\sigma_2(h)=\frac{2}{(N-1)(N-2)}\sum_{1\leq i<j\leq N-1}{k_i k_j}$ where $k_i$'s are the principal curvatures of $\Gamma$. The formula \eqref{dueM} is obtained in \cite{CP} by making a variation of the formulas \eqref{Muno}-\eqref{unoM} along the normal direction $\nu^\Gamma$ and differentiating along this direction (we recall that in our assumptions the vector $\nu^\Gamma$ is $Lip$-smooth up to $\de\Gamma$). By using \eqref{unoM}, \eqref{dueM}, the fact that $H$ is a positive constant, and the arithmetic-geometric inequality $\sigma_2(h)\leq H^2$, we get
\begin{equation}\label{JJl}
0=H\int_\Gamma \left(1 - H\left\langle x,\nu^\Gamma\right\rangle\right){\rm d}\sigma=\int_\Gamma \left(H - H^2\left\langle x,\nu^\Gamma\right\rangle\right){\rm d}\sigma \leq \int_\Gamma \left(H - \sigma_2(h)\left\langle x,\nu^\Gamma\right\rangle\right){\rm d}\sigma = 0.
\end{equation}
We explicitly remark that we have exploited in the previous inequality the strict starshapedness of $\Gamma$, i.e. $\left\langle x,\nu^\Gamma\right\rangle >0$. The relation \eqref{JJl} shows that it holds in fact the equality case in the arithmetic-geometric inequality $\sigma_2(h)\leq H^2$, which says that the second fundamental form $h$ is at every point of $\Gamma$ a multiple of the identity and so $h=H\mathbb{I}_{N-1}$. It is then a classical fact that such umbilicality property implies that $\Gamma$ is a portion of a sphere (see e.g. \cite[Section 5]{PT} for the details) as claimed in \eqref{claimJel}.\\
Once \eqref{claimJel} is proved, we can use Proposition \ref{eccop0perpolar} to infer that $p_0=O$, i.e. $\Gamma$ is a spherical sector. The case of the half-space cannot occur since $\omega$ is strictly contained in $\SNp$. We stress that, by invoking Proposition \ref{eccop0perpolar}, we are using again the starshapedness and the orthogonality assumptions. The proof is then complete.
\end{proof}

\section{Isoperimetric cones and symmetrization}\label{sec3}

\noindent We start by defining isoperimetric cones.

\begin{definition}\label{defisocone}
We say that $\Sw$ is an isoperimetric cone if the only sets contained in $\Sw$ which minimize the relative (to $\Sw$) perimeter under a volume constraint are the spherical sectors. This can be equivalently expressed saying that for any measurable set $E\subset\Sw$ with $|E|<+\infty$ the following isoperimetric inequality holds
\begin{equation}\label{iso}
P_\omega(E)\geq N\omega_N^{\frac{1}{N}}|E|^{\frac{N-1}{N}}
\end{equation}
ad equality is achieved if and only if $E$ is a spherical sector $\Swr$, $R>0$.\\
In \eqref{iso} $P_\omega$ is the relative perimeter of $E$ in $\Sw$ and $\omega_N=|S_{\omega, 1}|$.
\end{definition}

\noindent In \cite{LP} it has been proved that any smooth convex cone is isoperimetric (see \cite{CRS, FI, RitRos} for alternative proofs), in the trivial case when $\Sw$ is an half-space this holds up to translation. We observe that the proof of \cite{FI} also holds for nonsmooth convex cones.\\Here we show that any $C^{1,1}$-smooth cone sufficiently close to an isoperimetric cone, with respect to the $C^{1,1}$-distance on the sphere, is also isoperimetric.\\ More precisely, for $\eta>0$, let us consider the spherical cap
$$\Speta=\left\{ x=(x_1,\ldots,x_N)\in\SN\,:\, x_N>\eta \right\}$$
and (as in \cite{BF}) let us define the class of uniform $C^{1,1}$ open sets on the sphere:
\begin{definition}
Given $\eta>0$, $r>0$, we denote by $\Petar$ the class of open sets $\omega \subset\SN$ such that $\omega\subset\subset\Speta$ and for every $x\in\de\omega$ there exists a ball $B^+_r\subset \omega$ and a ball $B^-_r\subset \SN\cap\left(\SN\smallsetminus \omega\right)$ both of radius $r$, such that $x\in\de B^+_r\cap \de B^-_r$.
\end{definition}

The previous definition means that at every point  of $\de\omega$ an interior and an exterior ball condition holds and the radius of the balls can be taken equal to $r>0$, for all points in $\de\omega$.\\
In the sequel, for $\omega,\omega'\subset \SN$, we denote by ${\rm d}_{L^{\infty}}(\de\omega,\de\omega')$ the Hausdorff distance between $\de\omega$ and $\de\omega'$, with respect to the intrinsic metric on the sphere.

\begin{theorem}\label{thsect3}
Let $N\geq 3$ and let $\Sw$ be a isoperimetric cone belonging to $\Petar$ for some $\eta, r>0$. Then there exists $\e>0$ such that for any domain $\omega'\in\Petar$ with ${\rm d}_{L^\infty}(\de\omega,\de\omega')<\e$ the corresponding cone $\Swp$ is also isoperimetric.
\end{theorem}

\begin{proof}
Let us argue by contradiction and assume that there exists a sequence of cones $\Swpn$, with $\omega'_n\in\Petar$, such that
$$d_{L^\infty}(\de\omega,\de\omega'_n)\rightarrow 0\quad\mbox{as }n\rightarrow\infty \qquad\mbox{ but }\Swpn \mbox{ are not isoperimetric.}$$
Note that by \cite[Section 3]{RitRos}, the minimizers of the relative perimeter $P_{\omega'_n}$ with a volume constraint exist, for all fixed volume, because $\omega'_n\subset\subset\SNp$. Since we are assuming that $\Swpn$ are not isoperimetric cones, there exists a sequence of sets $E_n\subset\Swpn$ such that $E_n$ is not a spherical sector $S_{\omega'_n, R}$ though $E_n$ minimizes $P_{\omega'_n}$ under a volume constraint which, by the invariance under rescaling, we can assume to be $|E_n|=1$.\\
Now we follow the first part of the proof of Theorem 1.2 of \cite{BF} to deduce that $\de E_n\cap \Swpn$ are $C^1$-graphs over $\omega'_n$ (i.e. strictly starshaped with respect to $O$).\\First we observe that the sets $E_n$ are almost minimizers for the perimeter functional $P_{\omega'_n}$ (or $(\Lambda,r_0)$-perimeter minimizers in the sense of Almgren, for some $\Lambda\geq0$, $r_0>0$), see \cite[Section 21]{M} or \cite[Definition 1.8]{DPM}. Then, using the arguments of \cite[Theorem 3.4]{RitRos} and \cite[Proposition 21.13 and Theorem 21.14]{M} we get the existence of a set of finite perimeter $E_*\subset\RN$ such that, up to a subsequence,
\begin{eqnarray*}
&&|E_n\Delta E_*|\rightarrow 0\quad\mbox{ as }n\rightarrow\infty,\qquad\mbox{and}\\
&&P_{\Sw}(E_*)\leq\liminf_{n\rightarrow\infty}{P_{\Swpn}(E_n)}.\nonumber
\end{eqnarray*}
By standard arguments we then have that $E_*$ is a minimizer for $P_{\Sw}$ with the volume constraint $|E_*|=1$. Thus, $E_*$ is a spherical sector $\Swr$ for some $R>0$ because we are assuming that $\Sw$ is an isoperimetric cone. By the regularity theory for almost minimizers, both in the interior \cite[Part III]{M} and up to the boundary \cite{DPM}, we get that $\de E_n\cap \Swpn$ is a $C^{1,\gamma}$ ($\gamma\in (0,\frac{1}{2})$) manifold in a neighborhood of any $x\in \left(\de E_n\cap \Swpn\right)\smallsetminus \Sigma_{1,n}$, with $\Sigma_{1,n}$ closed and ${\rm H}^{N-1}\left(\Sigma_{1,n}\right)=0$, while for every $x\in \left(\de E_n\cap \de\Swpn\right)\smallsetminus \Sigma_{2,n}$, with $\Sigma_{2,n}$ closed, the closure of $\de E_n\cap \Swpn$ is a  $C^{1,\frac{1}{2}}$-manifold and ${\rm H}^{N-2}\left(\Sigma_{2,n}\right)=0$.\\
Now we want to use the closeness of $E_n$ to the smooth set $E_*=\Swr$ to show that the singular sets $\Sigma_{1,n}$ and $\Sigma_{2,n}$ are empty for $n$ sufficiently large.\\
To do this we use the characterization of the singular sets by the spherical excess (see \cite[Section 22]{M} and \cite[Section 3]{DPM}) which essentially asserts that at any singular point the spherical excess must be bigger than a constant $\delta>0$ which depends only on the dimension $N$. Then, using the continuity of the excess with respect to the ${\rm L}^1$-convergence of the almost minimizers (\cite[Section 22]{M}, \cite[Remark 3.6]{DPM}), the convergence of sequences of points $x_n\in\overline{E}_n$ to points in $\overline{E}_*$ (\cite[Theorem 21.14]{M}, \cite[Theorem 2.9]{DPM}), and the fact that $\overline{E}_*$ does not have singular points, we get that $\Sigma_{1,n}$ and $\Sigma_{2,n}$ are empty, for $n$ sufficiently large.\\
Finally, by the convergence of the outer unit normals (\cite[Theorem 26.6]{M}) we deduce that $\overline{E_n\cap\Swpn}$ is a strictly starshaped $C^1$ hypersurface, i.e. is a $C^1$-polar graph. Higher regularity then follows by standard elliptic regularity theory.\\
On the other side, since $E_n$ are minimizers for $P_{\omega'_n}$ with a volume constraint, by \cite{RitRos}-\cite{SZ} we know that $\de E_n$ has constant mean curvature and intersects $\de\Swpn$ orthogonally. Hence, by Theorem \ref{prth1} we have that $\de E_n$ must be a portion of a sphere centered at the origin, in other words $E_n$ is a spherical sector $S_{\omega'_n, R}$ for some $R>0$ which gives a contradiction.
\end{proof}

\begin{corollary}
The set of isoperimetric cones in $\Petar$ is an open set with respect to the $C^{1,1}$-distance of the boundaries.
\end{corollary}

\begin{remark}
Theorem \ref{thsect3} is a generalization of \cite[Theorem 1.2]{BF} where it is proved that almost convex cones are isoperimetric. We do not require the limit cone to be convex since we conclude by using Theorem \ref{prth1} which does not require the convexity of the cone. The use of Theorem \ref{prth1} also allows to shorten the proof of \cite{BF} as explained in the Introduction.
\end{remark}

\begin{remark}
As a consequence of Theorem \ref{prth1} we have that, if a cone contained in the hemisphere is not isoperimetric then a smooth, volume constrained minimizer $F$ cannot be a strictly starshaped set with respect to the vertex of the cone. This is because $\de F\cap \Sw$ must have constant mean curvature and intersect $\de\Sw$ orthogonally, so Theorem \ref{prth1} would give a contradiction because then $F$ would be a spherical sector $\Swr$.
\end{remark}

Now we define the $\o$-symmetrization for functions defined in sector-like domains in isoperimetric cones. This symmetrization was introduced in \cite{PaTr} for more general domains in $\RN$, $N\geq 3$, and in \cite{Ba} for $N=2$ (see also \cite{LiPaTr}).\\
Let $\Sw$ be an isoperimetric cone and $\O\subset\Sw$ a sector-like domain. For a measurable function $u:\O\longrightarrow\R$, we denote by $\mu(t)$ its distribution function
$$\mu(t)=|\{x\in\O\,\:\,|u(x)|>t\}|,\qquad t\in[0,+\infty),$$
and by $u^{\sharp}$ the decreasing rearrangement
$$u^{\sharp}(s)=\inf{\{t\geq 0\,\:\,\mu(t)<s\}},\qquad s\in[0,|\O|].$$
Then, for $R>0$, consider a spherical sector $\Swr=\Sw\cap B_R$ where $B_R$ is the ball centered at $O$ (the vertex of the cone) with radius $R$ and denote by $S_\o(\O)$ the spherical sector having the same measure as $\O$. The $\o$-symmetrization is defined as the transformation which associates to a function $u$ the radial decreasing function $u^*_\o(x)$ defined as:
$$u^*_\o(x)=u^\sharp(\o_N|x|^N),\qquad x\in S_\o(\O),$$
where $\o_N$ is the measure of the unit spherical sector $S_{\o,1}$. As pointed out in \cite{PaTr}, this symmetrization has the same properties as the Schwarz symmetrization. In particular:
\begin{equation}\label{normeLp}
\int_\O{|u(x)|^p\,{\rm d}x}=\int_{S_\o(\O)}{|u^*_\o(x)|^p\,{\rm d}x}\qquad\forall p>0.
\end{equation}
Now we consider the Sobolev space 
$$\Wup=\{u\in W^{1,p}(\O)\,\mbox{ such that }\,u=0\,\mbox{ on }\Gamma\},\qquad p\geq 1.$$
Let us observe that if $u\in\Wup$ then $u^*_\o=0$ on $\tilde{\Gamma}=S_\o(\O)\cap\Sw$. However it is not obvious that $u^*_\o\in W_0^{1,p}(S_\o(\O)\cup \tilde{\Gamma}_1)$, where $\tilde{\Gamma}_1=\de S_\o(\O)\cap \de\Sw$. If the cone is isoperimetric then this is true and actually an analogous of the Polya-Szego inequality holds and also the equality case can be completely characterized. More precisely we have:
\begin{theorem}\label{thPSZ}
Let $\Sw$ be an isoperimetric cone, and $\O\subset\Sw$ a sector-like domain. Let $1\leq p<\infty$ and $u\in\Wup$ with $u\geq 0$. Then
\begin{equation}\label{polsz}
\int_{S_\o(\O)}{|\nabla u^*_\o(x)|^p\,{\rm d}x}\leq \int_{\O}{|\nabla u(x)|^p\,{\rm d}x}.
\end{equation}
In particular $u^*_\o\in W_0^{1,p}(S_\o(\O)\cup\tilde{\Gamma}_1)$. Moreover equality holds in \eqref{polsz} if and only if $\O=S_\o(\O)$ and $u=u^*_\o$.
\end{theorem}
\begin{proof}
The statement \eqref{polsz} is already included in \cite[Proposition 1.2]{LiPaTr} without any proof, since it can be obtained by the same proof of the analogous inequality for the Schwarz symmetrization just replacing the classical isoperimetric inequality by \eqref{iso} everywhere. We refer to \cite{Ta} and to the book \cite[Theorem 2.3.1]{Ke}. If equality holds in \eqref{polsz} then, following for example the detailed proof of Theorem 2.3.1 in \cite{Ke}, it is easy to see that almost all level sets $E_t=\{u>t\}$, $t\geq 0$, of $u$ must satisfy the equality in the isoperimetric inequality in \eqref{iso}. Then, since the cone is isoperimetric, this implies that the level sets form a decreasing family of concentric spherical sectors. This implies that $\O=S_\o(\O)$ and $u=u^*_\o$.
\end{proof}
\begin{remark}
Let us point out that if we consider the classical Polya-Szego inequality, i.e. \eqref{polsz} in the space $W_0^{1,p}(\O)$ using Schwarz symmetrization, then it is not true that the equality case holds if and only if $\O$ is a ball and $u\equiv u^*$ (being $u^*$ the Schwarz symmetrization of $u$). Indeed, see \cite[Section 2.3]{Ke}, though one can easily deduce that almost all level sets of $u$ are balls, it can happen that the centers of them are different (see Example 3.1 in \cite{Ke}) so that $u\not\equiv u^*$. A remarkable result of \cite{BZ} shows that this can be prevented by assuming that the set of the points where $\nabla u^*(x)$ vanishes has zero measure. In the case of the $\o$-symmetrization this difficulty does not arise since the vertex of the cone $\Sw$ is fixed and the optimal sets for \eqref{iso} are spherical sectors with the same center.
\end{remark}

\section{Saint-Venant type principle}\label{sec4}

We want to exploit the properties of the $\o$-symmetrization in isoperimetric cones in order to minimize a suitable torsional energy. In this way we are going to prove the analogous, in our conical setting, of the classical Saint-Venant principle. We will see in Proposition \ref{domder} that this is closely related to the partially overdetermined problem \eqref{serrintype} in sector-like domains studied in \cite{PT}.\\
Fix an isoperimetric cone $\Sw$. For any sector-like domain $\O\subset \Sw$ we can define
\begin{equation}\label{defT}
\TwO=\inf_{v\in\Wud,\,v\neq 0 }-\frac{\left(\int_\O{v(x)\,{\rm d}x}\right)^2}{2\int_\O{|\nabla v(x)|^2\,{\rm d}x}}
\end{equation}
where $\Wud$ denotes the Sobolev space of functions in $W^{1,2}(\O)$ whose trace vanishes on $\Gamma$ (recall that $\de\O=\Gamma\cup\Gamma_1\cup\de\Gamma$). The functional $\TwO$ is well-defined by the Poincar\'e inequality, which holds true in $\Wud$ (see, e.g., \cite[Remark 2.3.3]{Ke1}). On the other hand, one can rewrite \eqref{defT} as
\begin{equation}\label{defTequiv}
\TwO=\inf_{v\in\Wud}{J(v)},
\end{equation}
where 
$$J(v)=\frac{1}{2}\int_\O{|\nabla v(x)|^2\,{\rm d}x}-\int_\O{v(x)\,{\rm d}x}.$$ 
Since $J$ is convex, it attains its unique minimum at the unique weak solution $u=u_\O$ of the mixed boundary value problem
\begin{equation}\label{mixedtype}
\begin{cases}
   -\Delta u=1 & \mbox{ in }\O, \\
   u=0 & \mbox{ on }\Gamma,\\
	\frac{\de u}{\de\nu}=0 & \mbox{ on }\Gamma_1\smallsetminus\{O\}.
\end{cases}
\end{equation}
Such a solution is positive in $\O$ and we have
$$\int_\O{u_\O(x)\,{\rm d}x}=\int_\O{|\nabla u_\O(x)|^2\,{\rm d}x}.$$
Therefore, we get
\begin{equation}\label{altradefT}
\TwO=-\frac{1}{2}\int_\O{|\nabla u_\O(x)|^2\,{\rm d}x}=-\frac{1}{2}\int_\O{u_\O(x)\,{\rm d}x}.
\end{equation}
\begin{remark}
The fact that \eqref{defT} and \eqref{defTequiv} are equivalent follows from the fact that, for all $v\in\Wud$ with $v\neq 0$, we have
$$J(v)=-\frac{\left(\int_\O{v(x)\,{\rm d}x}\right)^2}{2\int_\O{|\nabla v(x)|^2\,{\rm d}x}}+\frac{1}{2}\left(\left(\int_\O{|\nabla v(x)|^2\,{\rm d}x}\right)^{\frac{1}{2}}-\frac{\int_\O{v(x)\,{\rm d}x}}{\left(\int_\O{|\nabla v(x)|^2\,{\rm d}x}\right)^{\frac{1}{2}}}\right)^2.$$
Hence, if the infimum of the functional in \eqref{defT} (which is homogeneous of degree $0$) is attained at some function $w\in\Wud$ with $w\neq 0$, it is also attained at $\bar{w}=\lambda w$ with $\lambda=\frac{\int_\O{w}}{\int_\O{|\nabla w|^2}}$. Since $\int_\O{\bar{w}}=\int_\O{|\nabla \bar{w}|^2}$ we have
$$J(v)\geq -\frac{\left(\int_\O{v(x)\,{\rm d}x}\right)^2}{2\int_\O{|\nabla v(x)|^2\,{\rm d}x}}\geq -\frac{\left(\int_\O{\bar{w}(x)\,{\rm d}x}\right)^2}{2\int_\O{|\nabla \bar{w}(x)|^2\,{\rm d}x}}=J(\bar{w})\qquad\forall v\in\Wud,\,v\neq 0.$$
On the other hand, if the infimum of $J$ is attained at the non-null function $u\in\Wud$, then for any $v\in\Wud$ with $v\neq 0$ we can consider $\bar{v}=\lambda v$ where $\lambda= \frac{\int_\O{v}}{\int_\O{|\nabla v|^2}}$. As before, we have $\int_\O{\bar{v}}=\int_\O{|\nabla \bar{v}|^2}$ and we thus get
$$-\frac{\left(\int_\O{v(x)\,{\rm d}x}\right)^2}{2\int_\O{|\nabla v(x)|^2\,{\rm d}x}}=-\frac{\left(\int_\O{\bar{v}(x)\,{\rm d}x}\right)^2}{2\int_\O{|\nabla \bar{v}(x)|^2\,{\rm d}x}}=J(\bar{v})\geq J(u)\geq -\frac{\left(\int_\O{u(x)\,{\rm d}x}\right)^2}{2\int_\O{|\nabla u(x)|^2\,{\rm d}x}}.$$
\end{remark}

\vskip 0.3cm

The goal is to minimize $\TwO$ in the class of sector-like domains with a volume constraint.
We then define
$$\mathcal{C}_\o=\{\O\subset \Sw\,:\,\O\mbox{ is a sector-like domain with }|\O|=1\}.$$
We want to characterize
\begin{equation}\label{problemtomin}
\inf_{\O\in\mathcal{C}_\o}{\TwO}.
\end{equation}
\begin{remark}
As in the isoperimetric problem \eqref{iso} there is a natural invariance by rescaling. This is due to the fact that a dilated sector-like domain $t\cdot \O$ is still a sector-like domain for any $t>0$, and the functions in $\Wud$ are in natural correspondence with the functions in $W_0^{1,2}(t\cdot\O\cup\left(t\cdot\Gamma_{1}\right))$. This allows, as in the classical Saint-Venant problem, to write $\mathcal{T}_\o(|\O|^{-\frac{1}{N}}\cdot\O)=|\O|^{-\frac{N+2}{N}}\TwO$ and to reformulate the minimization problem in \eqref{problemtomin} as $\inf\left\{|\O|^{-\frac{N+2}{N}}\TwO\,:\,\O\mbox{ is a sector-like domain contained in }\Sw\right\}$ or in the alternative form $\inf\left\{\TwO\,:\,\O\mbox{ is a sector-like domain contained in }\Sw\mbox{ with }|\O|\leq 1\right\}$.
\end{remark}

Given $\O \in\mathcal{C}_\o$, we say that $\O_t=\phi_t(\O)$ is a volume preserving deformation of $\O$ if $|\O_t|=|\O|=1$ for $t$ small and $\phi_t$ is a one-parameter group of diffeomorphisms associated with a smooth vector field $V$ (which we can think with compact support) such that $V(x)\in T_x\de\Sw$ for all $x\in\de\Sw\smallsetminus \{O\}$ and $V(O)=0$. In particular $\O_t\in \mathcal{C}_\o$ for $t$ small: in this case we use the notations $\Gamma^t$ and $\Gamma_1^t$ respectively for $\de\O_t\cap \Sw$ and $\de\O_t\smallsetminus\overline{\Gamma}^t$.\\
We then say that $\O\in\mathcal{C}_\o$ is stationary (or critical point) for $\mathcal{T}_\o$ under the volume constraint if
$$\frac{d}{dt}_{|t=0}\mathcal{T}_\o(\O_t)=0$$
for every volume preserving deformation.\\
In the next proposition we characterize the stationary points of $\TwO$ via the domain-derivative technique, as for other similar problems \cite{HP, LS, Sim, SoZo}.

\begin{proposition}\label{domder}
Let $\Sw$ be any cone such that $\de\Sw\smallsetminus\{O\}$ is smooth. Consider $\O\in \mathcal{C}_\o$ having a smooth relative boundary $\Gamma$ with smooth $\de\Gamma\subset \de\Sw\smallsetminus\{O\}$, and assume that the unique weak solution $u_\O$ of \eqref{mixedtype} belongs to $W^{1,\infty}(\O) \cap W^{2,2}(\O)$. Then, $\O$ is a stationary point for $\mathcal{T}_\o$ under the volume constraint if and only if $u_\O$ satisfies the overdetermined condition $|\nabla u_\O|\equiv constant$ on $\Gamma$.
\end{proposition}
\begin{proof}
Consider any volume preserving deformation $\O_t$, which is determined by the vector field $V$ as above, so that $\O_t\in \mathcal{C}_\o$ for $t\in(-\delta,\delta)$, for some $\delta>0$. The fact that the volume is preserved implies
\begin{equation}\label{derivolum}
0=\frac{d}{dt}_{|t=0}|\O_t|=\int_{\de\O}{\left\langle V,\nu\right\rangle{\rm d}\sigma}=\int_{\Gamma}{\left\langle V,\nu\right\rangle{\rm d}\sigma},
\end{equation}
where in the last equality we used that $V$ is smooth and is tangent to $\de\Sw$ at every point of $\de\Sw\smallsetminus\{O\}$. On the other hand, we can consider the weak solution $u_t$ relative to the mixed boundary value problem \eqref{mixedtype} in $\O_t$. Since we have
$$\Wud=\{v\circ \phi_t\,:\, v\in W_0^{1,2}(\O_t\cup\Gamma^t_1)\},$$
we can consider
$$\hat{u}_t=u_t\circ \phi_t \in \Wud.$$
The fact that $u_t$ is a solution can be expressed as
$$\int_{\O_t}{\left\langle \nabla u_t(x),\nabla v(x)\right\rangle\,{\rm d}x} -\int_{\O_t}{v(x)\,{\rm d}x}=0\qquad\forall v\in W_0^{1,2}(\O_t\cup\Gamma^t_1).$$
We can transfer this relation on $\hat{u}_t$ as follows:
$$\int_{\O}{\left\langle M_t\nabla  \hat{u}_t(x),\nabla w(x)\right\rangle J_t(x) \,{\rm d}x}-\int_{\O}{w(x) J_t(x)\,{\rm d}x}=0\qquad\forall w\in \Wud,$$
where $J_t(x)={\rm det}(\mathcal{J}\phi_t(x))$ and $M_t=\mathcal{J}\phi^{-1}_t (\phi_t(x))\left(\mathcal{J}\phi^{-1}_t (\phi_t(x))\right)^T$. Let us now consider
$$F:(-\delta,\delta)\times \Wud \longrightarrow \left(\Wud\right)^*$$
defined as
$$F(t,v)= -{\rm div}\left(M_t\nabla v\right) - J_t.$$
We know that $F(t,\hat{u}_t)=0$ for every $t$. One can show that $F$ is smooth and the Gateaux derivative $\de_v F(0,u_\O)=-\Delta v$ (this defines an isomorphism since for every $f\in \left(\Wud\right)^*$ there exists a unique $v\in \Wud$ such that $-\Delta v=f$). Therefore $t\mapsto \hat{u}_t$ is smooth, and then
$$t\mapsto u_t=\hat{u}_t\circ \phi_t^{-1} \quad\mbox{is differentiable}.$$
We denote by $u'$ and $\hat{u}'$ respectively the derivatives with respect to $t$ of $u_t$ and $\hat{u}_t$ computed at $t=0$. We have that
\begin{equation}\label{relazuprimo}
u'=\hat{u}'-\left\langle \nabla u_\O, V\right\rangle.
\end{equation}
Since $\overline{\Omega}$ is smooth except for $\de\Gamma$ and the vertex $O$, so are $u_\O, u'$. We have then $\Delta u'=0$ in $\Omega$. Being $\hat{u}'\in \Wud$, we have also $u'=-\left\langle \nabla u_\O, V\right\rangle$ on $\Gamma$. Finally, since $\de\Sw$ is mapped into itself by $\phi_t$, the points in $\Gamma_1$ stay in $\Gamma_1$ for some small $t$. Differentiating in $t$ variable the relation
$$\left\langle \nabla u_t (\phi_t(x)), \nu_{\phi_t(x)}\right\rangle=0,$$
we get
$$0=\left\langle \nabla u' (x), \nu_{x}\right\rangle + V\left(\left\langle \nabla u, \nu\right\rangle\right)=\left\langle \nabla u' (x), \nu_{x}\right\rangle$$
for every $x\in\Gamma_1\smallsetminus\{O\}$. Hence $u'$ is a solution to the mixed boundary value problem
$$\begin{cases}
   -\Delta u'=0 & \mbox{ in }\O, \\
   u'=-\frac{\de u_\O}{\de\nu}\left\langle V,\nu\right\rangle & \mbox{ on }\Gamma,\\
	\frac{\de u'}{\de\nu}=0 & \mbox{ on }\Gamma_1\smallsetminus\{O\}.
\end{cases}$$
We can also compute the derivative with respect to $t$ of the torsion functional. From \eqref{altradefT} we get
\begin{eqnarray}\label{variazprimatorsione}
&&\frac{d}{dt}_{|t=0}\mathcal{T}_\o(\O_t)=-\frac{1}{2}\int_\O{ \hat{u}'(x)\,{\rm d}x}-\frac{1}{2}\int_{\O}{ u_\O(x){\rm div}(V)(x)\,{\rm d}x}\nonumber\\
&&=-\frac{1}{2}\int_\O {u'(x)\,{\rm d}x}-\frac{1}{2}\int_{\O}{ {\rm div}(u_\O V)(x)\,{\rm d}x}\nonumber\\
&&=+\frac{1}{2}\int_\O{ u'(x)\Delta u_\O(x)\,{\rm d}x}=-\frac{1}{2}\int_\O{ \left\langle \nabla u'(x),\nabla u_\O(x)\right\rangle\,{\rm d}x}+\frac{1}{2}\int_\O{ {\rm div}\left(u'\nabla u_\O\right)(x)\,{\rm d}x}\nonumber\\
&&= -\frac{1}{2}\int_\O{ {\rm div}(u_\O\nabla u')(x)\,{\rm d}x}-\frac{1}{2}\int_\Gamma |\nabla u_\O|^2\left\langle V,\nu \right\rangle\nonumber{\rm d}\sigma\\
&&=-\frac{1}{2}\int_\Gamma |\nabla u_\O|^2\left\langle V,\nu \right\rangle{\rm d}\sigma.
\end{eqnarray}
To justify the previous applications of the divergence theorem we can make use of \cite[Lemma 2.1]{PT}, which requires a certain degree of integrability for the relevant vector fields. The assumption that $u_{\O}\in W^{1,\infty}(\O) \cap W^{2,2}(\O)$ is sufficient for these purposes. We notice in particular that, under this assumption, from \eqref{relazuprimo} we have $u'\in W^{1,2}(\O)$.\\
The desired statement then follows from \eqref{variazprimatorsione} and \eqref{derivolum}. As a matter of fact, if $u_\O$ satisfies the overdetermined condition $|\nabla u_\O|\equiv constant$ on $\Gamma$, then it is now obvious that $\O$ is a stationary point for $\mathcal{T}_\o$. Viceversa, if we have a stationary point for $\mathcal{T}_\o$, then $\int_\Gamma \left(|\nabla u_\O|^2-c\right)\left\langle V,\nu \right\rangle=0$ for all constants $c$ and any $V$ admissible and satisfying \eqref{derivolum} (see e.g. \cite{SZ} for the construction of volume preserving deformations starting from admissible vector fields with the property \eqref{derivolum}). If we assume by contradiction that $|\nabla u_\O|$ is not constant on $\Gamma$, we could then find a compact set $K$ included in $\Gamma$ where $|\nabla u_\O|$ is not constant. We could then pick a nonnegative cut-off function $\psi$ which is $1$ on $K$ and with support compactly contained in the cone, and we could choose $c=\frac{1}{\int_\Gamma \psi}\int_{\Gamma}\psi|\nabla u_\O|^2$ and build a deformation starting from $V=\psi(|\nabla u_\O|^2-c)\nu$: the stationary condition would then imply that $\int_K \left(|\nabla u_\O|^2-c\right)^2=0$, giving a contradiction.
\end{proof}

From Proposition \ref{domder} and Theorem A we deduce the following
\begin{corollary}
Let $\Sw$ be a convex cone such that $\de\Sw\smallsetminus\{O\}$ is smooth. Consider $\O\in \mathcal{C}_\o$ having a smooth relative boundary $\Gamma$ with smooth $\de\Gamma\subset \de\Sw\smallsetminus\{O\}$, and assume that the unique weak solution $u_\O$ of \eqref{mixedtype} belongs to $W^{1,\infty}(\O) \cap W^{2,2}(\O)$. If $\O$ is a stationary point for $\mathcal{T}_\o$ under the volume constraint, then $\O=\Sw\cap B_{R}(p_0)$ and one of the following two possibilities holds
\begin{itemize}
\item[(i)] $p_0=O$;
\item[(ii)] $p_0\in\de\Sw$ and $\Gamma$ is a half-sphere lying over a flat portion of $\de\Sw$.
\end{itemize}
\end{corollary}

We remark that, by a direct computation, one can see that in this situation we have $\mathcal{T}_\o(\Sw\cap B_{R}(O))<\mathcal{T}_\o(\mbox{half-ball})$ (unless $\Sw$ is an half-space). Therefore, we know that $\Sw\cap B_{R}(O)$ are the smooth minimizers for the torsional function under the hypotheses of the previous corollary.

\vskip 0.4cm

We are finally going to prove that in isoperimetric cones (not just in smooth convex cones) we can always characterize the minimum point for $\mathcal{T}_\o$ under the volume constraint (with no additional assumption on the smoothness of the competitor $\O$, nor on the summability of the related $u_\O$).

\begin{theorem}\label{saintV}
Let $\Sw$ be a isoperimetric cone. Then the spherical sector $\Swr$ with $|\Swr|=1$ is the unique minimizer for $\mathcal{T}_\o$ under the volume constraint, i.e. we have
$$\TwO\geq \mathcal{T}_\o(S_\o(\O))\qquad \forall\O\in \mathcal{C}_\o$$
and equality holds if and only if $\O=S_\o(\O)$.
\end{theorem}
\begin{proof}
This is a consequence of Theorem \ref{thPSZ}. As a matter of fact, for any sector-like domain $\O\subset\Sw$ we have
\begin{eqnarray*}
\TwO&=&\frac{1}{2}\int_\O{|\nabla u_\O(x)|^2\,{\rm d}x}-\int_\O{u_\O(x)\,{\rm d}x}\\
&\geq& \frac{1}{2}\int_{S_\o(\O)}{|\nabla \left(u_\O\right)^*_\o(x)|^2\,{\rm d}x} -\int_{S_\o(\O)}{\left(u_\O\right)^*_\o(x)\,{\rm d}x}
\geq \mathcal{T}_\o(S_\o(\O)),
\end{eqnarray*}
where in the first inequality we used \eqref{normeLp} and \eqref{polsz} (respectively with $p=1$ and $p=2$), and in the second inequality we used the fact that $\left(u_\O\right)^*_\o\in W_0^{1,2}(S_\o(\O)\cup\tilde{\Gamma}_1)$ (which is also a byproduct of Theorem \ref{thPSZ}; we recall that $u_\O\geq 0$).\\
The equality case follows from the equality case of \eqref{polsz}.
\end{proof}

\bibliographystyle{amsplain}

\end{document}